\documentclass[11pt]{article}

\usepackage{amsmath,amssymb,amsthm} 
\usepackage[unicode,breaklinks=true,colorlinks=true]{hyperref}
\usepackage[capitalise]{cleveref}
\usepackage[dvipsnames]{xcolor}
\usepackage{mathrsfs}
\usepackage{verbatim}
\usepackage{mathtools}
\usepackage[normalem]{ulem}

\usepackage{soul}
\usepackage{marginnote}

\usepackage[top=1in, bottom=1in, left=1in, right=1in, marginparwidth=1in, marginparsep=0.1in]{geometry}

\numberwithin{equation}{section}

\newtheorem{theorem}{Theorem}[section]

\newtheorem{lemma}[theorem]{Lemma}
\newtheorem{definition}[theorem]{Definition} 

\newtheorem{corollary}[theorem]{Corollary}

\theoremstyle{remark}
\newtheorem{remark}[theorem]{Remark}

\definecolor{darkblue}{rgb}{0,0,0.7}
\definecolor{darkred}{rgb}{0.6,0,0}

\newcommand{\cmp}[1]{{\color{Mahogany}{#1}}}

\newcommand{\snt}[1]{{\color[rgb]{0,0,1}$\bullet$}\vspace{-10pt}\marginpar{\color[rgb]{0,0,0.8}$\bullet$ \footnotesize #1}\vspace{10pt}}

\newcommand{\al}{\alpha}
\newcommand{\be}{\beta}

\newcommand{\ga}{{\gamma}}

\newcommand{\la}{\lambda}

\newcommand{\td}{\tilde}

\newcommand{\De}{\Delta}
\renewcommand{\th}{\theta}

\newcommand{\Bp}{{\dot B_{p,\I}^{-1+\frac 3 p}}}

\newcommand{\R}{{\mathbb R }}

\newcommand{\N}{{\mathbb N}}

\newcommand{\nb}{{\nabla}}

\newcommand{\I}{\infty}
 
\renewcommand{\div}{\mathop{\mathrm{div}}}

\newcommand{\supp}{\mathop{\mathrm{supp}}}

\newcommand{\donothing}[1]{{}}

\newcommand{\EQ}[1]{\begin{equation}\begin{split} #1 \end{split}\end{equation}}
\newcommand{\EQN}[1]{\begin{equation*}\begin{split} #1 \end{split}\end{equation*}}

\newcommand{\loc}{\mathrm{loc}}

\newcommand{\uloc}{\mathrm{uloc}}

\title{Estimation of non-uniqueness and short-time asymptotic expansions for Navier-Stokes flows}
\author{Zachary Bradshaw \and Patrick Phelps} 
\date{\today}

\begin{document}

\maketitle 
 
\begin{abstract}
There is considerable evidence that solutions to the non-forced 3D Navier-Stokes equations in the natural energy space are not unique. Assuming this is the case, it becomes important to quantify how non-uniqueness evolves. In this paper we provide an algebraic estimate for how rapidly two possibly non-unique solutions can separate over a compact spatial region in which the initial data has sub-critical regularity. Outside of this compact region, the data is only assumed to be in the scaling critical weak Lebesgue space and can be large. In order to establish this separation rate, we develop a new spatially local, short-time asymptotic expansion which is of independent interest. 
\end{abstract}

\section{Introduction}

The Navier-Stokes equations, 
\EQ{\label{eq.NS}
\partial_t u -\Delta u +u\cdot\nb u+\nb p = 0;\qquad \nb \cdot u=0,
}
govern the evolution of a viscous incompressible flow's velocity field $u$, which is a vector, and its associated scalar pressure $p$.
The system is supplemented with a divergence free initial datum $u_0$. We consider the problem on $\R^3\times (0,T)$, where $0<T\leq \I$.
A foundational mathematical treatment of the problem was given by Leray in \cite{leray} where global weak solutions were constructed for finite energy data. Solutions with the properties of those constructed by Leray are referred to as Leray weak solutions.
Recent work suggests that uniqueness does not hold in the class of Leray weak solutions. Indeed, non-uniqueness has been affirmed in weaker classes than the Leray class \cite{BuckmasterVicol} and within the Leray class for the forced Navier-Stokes equations \cite{AlBrCo}. Within the Leray class and with no forcing, a research program of Jia and \v Sver\' ak \cite{JS,JiaSverakIll}, as well as the numerical work of Guillod and \v Sver\' ak \cite{GuillodSverak}, supports non-uniqueness. This program first establishes non-uniqueness in a class of solutions with large data in the critical weak Lebesgue space---this is the type of solutions we presently consider and give a precise definition below. While the evidence suggests non-uniqueness, there is not a clear picture of how non-uniqueness would evolve.
In this note, we take the view that solutions are not unique and seek to quantify how rapidly distinct solutions can separate as they evolve from a shared initial state.
In particular, we are interested in the following question:
\begin{quote}
 How can non-uniqueness be quantified in terms of local properties of the initial data? 
\end{quote}
To answer this question we seek conditions so that, given some divergence free $u_0$, ball $B$, positive exponent $\sigma$, time $T>0$, and weak solutions $u_1$ and $u_2$ to \eqref{eq.NS} both evolving from $u_0$, we have
 \[
 \| u_1 - u_2\|_{L^\I(B)} (t)\lesssim t^{ \sigma},
 \]
for all $0<t<T$. We refer to bounds like the above as an ``estimation of non-uniqueness'' and the right-hand side as a ``separation rate.''

A preliminary perspective on this question follows from the local smoothing of Jia and \v Sver\' ak \cite{JS}.
Local smoothing says that, if $u_0$ is sufficiently regular in a ball $B$, then a solution $u$ remains regular on $B'\times [0,T]$ for some $T>0$, where $B'\Subset B$. This can be viewed as saying that the non-local effects of the pressure are not strong enough to overcome the local regularity of the data. Local regularity is proven in \cite{JS} by showing that, for solutions in the local Leray class,\footnote{This is a more general class than the Leray class and was introduced by Lemari\'e-Rieusset. See \cite{LR} and the later papers \cite{BT8,JS,KS,KwTs,KMT} for useful properties. Local Leray solutions are sometimes referred to as local energy solutions.} if $u_0|_B\in L^p (B)$ for some $3<p\leq \I$ and $U$ is the strong solution to the Navier-Stokes equations with initial data a divergence free localization of $u_0$ to $B$, then $u-U \in C^\ga_{\text{parabolic}} ( B'\times [0,T])$ where $\ga=\ga(p) \in (0,1)$. Note that in the definition of $C^\ga_{\text{parabolic}} ( B'\times [0,T])$, the exponent for the time variable is $\ga/2$. Since $U$ is uniquely determined by $u_0$, this implies that, for possibly distinct solutions $u_1$ and $u_2$ with the same data $u_0$, we have 
\EQ{\label{ineq.JSseparation}
 \| u_1- u_2\|_{L^\I(B')}(t)\leq \| u_1- U\|_{L^\I(B')}(t)+ \| U- u_2\|_{L^\I(B')}(t)\lesssim t^{\frac \ga 2}.
}
Because $\ga/2<1$, this allows a dramatic jump between the two solutions at small times. 

A stronger separation rate is identified for discretely self-similar solutions, i.e.~solutions satisfying $u^\la(x,t):=\la u(\la x,\la^2t)$, for some $\la >1$, with data in $L^p_\loc(\R^3\setminus \{0\})$ for $3<p\leq \I$ \cite{BP1}. There, due to the global scaling properties of the solution, 
\[
 |u_1-u_2|(x,t)\lesssim \frac {t^{\frac 3 2}} {(|x|+\sqrt t)^{4}},
\]
 outside of a space-time paraboloid $|x|\geq R_0\sqrt t$, for some $R_0\ge0$. Away from $x=0$, this gives the separation rate $t^{3/2}$, which is stronger than \eqref{ineq.JSseparation} for $t\leq 1$. Although we do not have a proof, we expect the rate $t^{3/2}$ is optimal because it arises in \cite{BP1} from pointwise bounds for gradients of the Oseen tensor \cite{VAS} which seem unavoidable. The solutions in \cite{BP1} have a great deal of structure due to their assumed scaling invariance and it is natural to seek separation rates under relaxed conditions. 
 
 In this paper, we almost recover the separation rate $t^{3/2}$ for DSS solutions with no scaling assumption on $u_0$. 
 We take our initial data to be in $L^{3,\I}(\R^3)$, which coincides with the weak Lebesgue space $L^3_w$ and is a Lorentz space.\footnote{$L^{3,\I}$ includes all the DSS data considered in \cite{BP1}. It is important in the analysis of the Navier-Stokes equations as an endpoint space where many desirable features such as regularity or uniqueness break down. For example, there is a time-local unique strong solution when $u_0$ is possibly large in $L^3$ \cite{Kato}, but this is unknown in the larger space $L^{3,\I}$. It is a \textit{critical} space in that it is scaling invariant with respect to the scaling of \eqref{eq.NS}.} 
 If, additionally, $u_0|_{B}\in L^p(B)$ for a ball $B$ and some $3<p\leq \I$, then we show there exists a time $T>0$ so that, for any $\sigma<3/2$, any two weak solutions $u_1$ and $u_2$ in a certain class satisfy
 \[
 \|u_1-u_2\|_{L^\I(B')}(t)\lesssim t^\sigma,
 \]
where $B'\Subset B$ and $0<t<T$. 
 This class of initial data is motivated by a natural type-I blow-up scenario wherein a strong solution $u$ defined on $\R^3 \times (-1,0)$ satisfying 
\[
 |u(x,t)|\lesssim \frac 1 {|x|+\sqrt {-t}},
\]
develops a singularity at the space-time origin. The singular profile would satisfy $$|u(x,0)|\lesssim |x|^{-1}\in L^{3,\I}.$$
Because uniqueness is not expected for large $L^{3,\I}$ data, upon singularity formation the solution might branch into distinct flows. In this scenario, our theorem provides an upper bound on how fast the branching solutions can separate away from the singularity. It is worth pointing out that our main theorem can handle singular profiles which are worse than $|x|^{-1}$. Related to this, the initial data in \cite{BP1} is only locally critical at the origin; it is locally sub-critical\footnote{We say the space $X$ is \textit{sub-critical} if $\|u_0\|_X = \la^\al \|u_0^\la \|_X$ where $\al>0$. Examples of critical spaces are $L^p$ for $p\in (3,\I]$. Typically, inclusion in sub-critical spaces controls small scales and leads to regularity. For \textit{super-critical} spaces, $\al<0$ and small scales are typically not controlled. For critical spaces, small scales are usually controlled to an extent when $C_c^\I$ is dense in the space. Critical spaces where this fails, like $L^{3,\I}$, are referred to as \textit{ultra-critical}.} everywhere else. In our   theorem, the only sub-critical assumption is within the ball $B$; the data can have $L^{3,\I}$ singularities anywhere else. 

Before stating our result we define the class of solutions we have in mind, which was originally developed by Barker, Seregin and \v Sver\' ak \cite{BaSeSv} and extends ideas in \cite{SeSv}. This notion of solution has since been extended to non-endpoint critical Besov spaces of negative smoothness \cite{AB}. 

\begin{definition}[Weak $L^{3,\I}$-solutions] Let $T>0$ be finite. Assume $u_0\in L^{3,\I}$ is divergence free. We say that $u$ and an associated pressure $p$ comprise a weak $L^{3,\I}$-solution if \begin{itemize}
 \item $(u,p)$ satisfies \eqref{eq.NS} distributionally,
 \item $u$ satisfies the local energy inequality of Scheffer \cite{VS76b} and Caffarelli, Kohn and Nirenberg \cite{CKN} for non-negative test functions compactly supported in $\R^3\times (0,T)$,
 \item for every $w\in L^2$, the function 
 \[
 t\to \int w(x)\cdot u(x,t)\,dx,
 \]
 is continuous on $[0,T]$,
 \item $\td u :=u-e^{t\Delta}u_0$ satisfies, for all $s\in (0,T)$,
 \EQ{\label{ineq.BSSbound}
 \sup_{0<s<t}\| \td u \|^2_{L^2} (s) +\bigg(\int_0^t \| \nb \td u\|_{L^2}^2(s)\,ds \bigg)^\frac 1 2 <\I,
 }
 and
\EQ{ 
 \| \td u\|_2^2 +2\int_0^t \int|\nb \td u|^2\,dx\,ds\leq 2\int_0^t \int ( e^{t\Delta}u_0 \otimes \td u + e^{t\Delta}u_0 \otimes e^{t\Delta}u_0) : \nb \td u\,dx\,ds.
 }
\end{itemize}

\end{definition}

In \cite{BaSeSv}, weak solutions are constructed which satisfy the above definition for all $T>0$. 
 Also, due to their spatial decay, weak $L^{3,\I}$-solutions are mild,\footnote{Although this can be proved directly, it also follows from \cite{BT7} or \cite{LR}.} which means they satisfy the formula
\[
u(x,t)=e^{t\Delta}u_0 - \int_0^te^{(t-s)\Delta} \mathbb P \nb \cdot (u\otimes u)\,ds,
\]
 where $\mathbb P$ is the Leray projection operator. The nonlinear contribution to $u$ is a bilinear operator which denote by
 \[
 B(f,g):= \int_0^te^{(t-s)\Delta} \mathbb P \nb \cdot (f\otimes g)\,ds.
 \]

An important observation in \cite{BaSeSv} is that the nonlinear part of a weak $L^{3,\I}$-solution satisfies a dimensionless energy estimate, namely
\EQ{\label{ineq.BSSdecay}
 \sup_{0<s<t}\| \td u \|_{L^2} (s) +\bigg(\int_0^t \| \nb \td u\|_{L^2}^2(s)\,ds \bigg)^\frac 1 2 \lesssim_{u_0} t^{\frac 1 4}.
}
We emphasize that the energy associated with $\td u$ vanishes at $t=0$. This decay property will be essential in our work. It appeared earlier in the \textit{a priori} estimates of the weak discretely self-similar solutions constructed in \cite{BT1} as well as \cite{SeSv}, which is the precursor to \cite{BaSeSv}. 
It is used in the Calderon-type splitting\footnote{See \cite{Calderon}.} construction in \cite{BaSeSv} to deplete a time singularity.
In \cite{AB}, it is established in Besov spaces with $e^{t\Delta}u_0$ replaced by higher Picard iterates, which are defined below.
Other dimensionless estimates can be derived from \eqref{ineq.BSSdecay}; indeed, below we prove that, for $q\in (3/2,3)$,
 \EQ{\label{ineq:newEstimate}
\| \td u\|_{L^r(0,T;L^q)}\lesssim_{q,u_0} T^{\frac 1 2}\quad \text{for} \quad r = \frac {2q} {2q-3}.
} Given \eqref{ineq.BSSdecay}, this is primarily of interest when $q<2$.

As pointed out in \cite{BaSeSv}, \eqref{ineq.BSSdecay} can be viewed as an estimate on the separation rate of two weak $L^{3,\I}$-solutions since, denoting two such solutions with the same data by $u_1$ and $u_2$, we have
\[
 \| u_1 - u_2\|_{L^2}(t)\lesssim \| \td u_1\|_{L^2}(t)+\|\td u_2\|_{L^2}(t)\lesssim t^{\frac 1 4}.
\]
Notably, this is a global estimate. For data in $L^{3,\I}$, global estimates are confined to super-critical norms since we do not expect $\td u$ to be in a stronger space than $u$---indeed, any singularity at a positive time is carried by $\td u$ not by $e^{t\Delta}u_0$. Such singularities can possibly occur at arbitrarily small times. Therefore, if we seek a finer estimate (i.e.~using a sub-critical norm) on the separation of the flows using the reasoning above, it should be confined to a local region where local smoothing holds, e.g.~where the initial data is sub-critical. The following theorem provides such an estimate.

\begin{theorem}[Estimation of non-uniqueness]\label{thrm.main}
Assume $u_0\in L^{3,\I}$ and is divergence free. Fix $x_0\in \R^3$. Assume that $u_0|_B\in L^p(B)$ where $B= B_2(x_0)$ and $p\in (3,\I]$. Let $u_1$ and $u_2$ be weak $L^{3,\I}$-solutions with data $u_0$. Then, there exists $T=T(p,u_0)>0$ so that, for every $\rho\in (0,1/2)$, $\sigma \in (0,3/2)$, and $t\in (0,T)$,
\[
 \| u_1-u_2\|_{L^\I (B_{\rho}(x_0))} (t)\lesssim_{\rho, p,\sigma,u_0} t^{\sigma}, 
\] 
where the dependence on $u_0$ is via the quantities $\|u_0\|_{L^p(B)}$ and $\|u_0\|_{L^{3,\I}}$.
\end{theorem}
Insofar as non-uniqueness in the Leray class is concerned, if $u_0\in L^2\cap L^{3,\I}$, then any weak $L^{3,\I}$-solution is also a Leray weak solution as discussed in \cite{BaSeSv}. Hence our result applies to a subset of the Leray class.

Theorem \ref{thrm.main} is a corollary of the following theorem, the proof of which constitutes the bulk of this paper. 
Before stating the theorem, we recall the definition of Picard iterates. Let $P_0 = P_0 (u_0) = e^{t\Delta}u_0$ and define the $k^{\text{th}}$ Picard iterate to be $P_k = P_0 - B(P_{k-1},P_{k-1})$. Classically, the Picard iterates converge to a solution to \eqref{eq.NS} whenever \eqref{eq.NS} can be viewed as a perturbation of the heat equation. This is not the case for large $L^{3,\I}$ data, so we do not expect convergence of $P_k$ to $u$ when $u$ is a weak $L^{3,\I}$-solution. Nonetheless, the Picard iterates do capture some asymptotics at $t=0$ of weak $L^{3,\I}$-solutions, which is the point of the following theorem. 

\begin{theorem}[Local asymptotic expansion]\label{thrm.main2}
Assume $u_0\in L^{3,\I}$ and is divergence free. Fix $x_0\in \R^3$ and $p\in (3,\I]$. Assume further that $u_0|_B\in L^p(B)$ where $B=B_2(x_0)$. Then, there exists $\ga =\ga (p)\in(0,1)$ and $T=T(p,\|u_0\|_{L^{3,\I}},\|u_0\|_{L^p(B)})>0$ so that for any $\rho\in (0,1/2)$, $\sigma \in (0,3/2)$, $t\in (0,T)$ and $k=0,1,\ldots, k_0$,
\[
 \| u-P_k\|_{L^\I (B_\rho(x_0))} (t)\lesssim_{\rho, p,u_0,\sigma,k} t^{a_k}, 
\] 
where $a_0 = \min\{\ga/2, 1/2-3/(2p)\}$,
$a_{k+1} =\min\left\{\sigma, k( 1/2 -3/(2p))+a_0\right\}$ and 
 $k_0$ is the smallest natural number so that 
\[
k_0\bigg(\frac 1 2 -\frac 3 {2p}\bigg) +a_0 \geq \sigma.
\]
In particular, $a_{k_0}=\sigma$ and $a_k>a_{k-1}$ for $k=1,\ldots,k_0$. It follows that, for $(x,t)\in B_{\rho}(x_0) \times (0,T)$, and letting $a_{-1}= -3/(2p)$, we have
\[
u(x,t)= P_0 + \sum_{k=0}^{k_0-1}\mathcal O (t^{a_k}) + \mathcal O(t^{\sigma}) = \sum_{k=-1}^{k_0}\mathcal O (t^{a_k}),
\]
where the $\mathcal O (t^{a_k})$ terms are exactly solvable for $-1\leq k<k_0$. 
\end{theorem}
 
Short-time asympototic expansions have been examined by Brandolese for small self-similar flows \cite{Brandolese} and by Brandolese and Vigneron for both small (in which case the expansion holds for all times) and large (in which case the data is globally sub-critical and the expansion is up to a finite time) non-self-similar flows \cite{BV}. A follow-up paper by Bae and Brandolese considers the {forced} Navier-Stokes \cite{BaeBrandolese}. In \cite{KR}, Kukavica and Ries give an expansion in arbitrarily many terms assuming the solution is smooth. In all of the preceding papers, either the initial data is strong enough to generate smooth solutions (e.g.~it is in a sub-critical class or is small in a critical class) or the solution is assumed to be smooth. Additionally, the terms of the asymptotic expansions depend on $u$.

The novelty of Theorem \ref{thrm.main2} is that it establishes time asymptotics without any scaling assumption (cf. \cite{BP1}) or requirements implying global regularity on the relevant  time domain (cf. \cite{BaeBrandolese,Brandolese,BV,KR}). The asymptotics depend only on $u_0$---they are independent of $u$ which is necessary for Theorem \ref{thrm.main}. Because ours is a spatially local expansion, spatial asymptotics are not relevant.

Long-time asymptotic expansions have also been studied extensively; see, for example, \cite{FS,FS2,GW,HM}, the review article \cite{BrandoleseSchonbek}, and the references therein. The spatial asymptotics for the stationary problem have also been studied; see, for example, \cite{KoSv} and the references therein. 

With Theorem \ref{thrm.main2} in hand, we quickly prove Theorem \ref{thrm.main}.
\begin{proof}[Proof of \cref{thrm.main}] Suppose $u_1$ and $u_2$ are weak $L^{3,\I}$-solutions with data $u_0$.
By \cref{thrm.main2}, we have for $i=1,2$ that
\[\|u_i-P_{k_0}\|_{L^\I(B_{\rho}(x_0))} (t) \lesssim_{\rho, p,\sigma,u_0} t^\sigma,\]
for all $0<t<T$. By the uniqueness of Picard iterates, we infer
\[
 \|u_1-u_2\|_{L^\I(B_{\rho}(x_0))}(t) \leq \|u_1-P_{k_0}\|_{L^\I(B_{\rho}(x_0))}(t)+\|u_2-P_{k_0}\|_{L^\I(B_{\rho}(x_0))}(t)\lesssim_{\rho, p,\sigma,u_0} t^{\sigma},
\]
for all $0<t<T$.
\end{proof}
\medskip 
\noindent \textbf{Discussion of the proof:} 
By local smoothing \cite{JS}, it is not difficult to show that
\[
 \| u-P_0\|_{L^\I(B_{\frac 1 2}(x_0))}(t)\lesssim t^{\frac \ga 2},
\]
for some $\ga=\ga(p) \in (0,1)$ and across some time interval. Our main insight is that this bound improves when $P_0$ is replaced by higher Picard iterates, a consequence of the self-improvement property of Picard iterates which has been used elsewhere, e.g.~\cite{AB,Brandolese,GIP}. To see how this works, we note that 
\EQ{\label{eq:splitting}
 u-P_{k+1} 
 = B(u-P_{k},u-P_{k}) + B(u-P_{k}, P_{k})+B(P_k, u-P_{k}).
}
Each term on the right hand side locally has an algebraic decay rate at $t=0$. The product structure and the time integral in the bilinear operator $B(f,g)$ lead to an improved algebraic decay rate for the left hand side compared to that for $u-P_k$. This improvement is only local. The far-field contributions to the flow are managed using a new \textit{a priori} bound for weak $L^{3,\I}$-solutions---see Corollary \ref{cor:decay}. The properties of weak $L^{3,\I}$-solutions \cite{AB,BaSeSv} are used critically throughout.

\bigskip\noindent \textbf{Organization:} In Section 2, we establish several key lemmas, most importantly the extension of the decay property \eqref{ineq.BSSdecay} to \eqref{ineq:newEstimate}. We also establish some elementary properties of Picard iterates. Section 3 contains the proof of Theorem \ref{thrm.main2}.
 
\section{Preliminaries}

In this section, we prove new \textit{a priori} bounds for weak $L^{3,\I}$-solutions. See Lemmas \ref{lemma.decay} and \ref{lemma.decay2}. We then establish a property of Picard iterates in Lemma \ref{lemma:Pkq}.


Due to scaling considerations, one predicts that if the energy-level quantities on the left hand side of \eqref{ineq.BSSdecay} are replaced by lower Lebesgue or Lorentz norms, then the exponent on the right hand side will increase to preserve the scaling of the inequality. With our application in mind, it is natural to ask if the following dimensionless estimate holds
\[
 \sup_{0<s<t}\| \td u \|_{L^{\frac 3 2,1}} (s) \lesssim t^{\frac 1 2}.
\]
This estimate is motivated by the $L^{ 3/2,1}$-$L^{3,\I}$ duality pairing.
There are barriers to establishing the above decay rate, but nearby rates are within reach as the next two lemmas show. In the lemmas, we replace $\td u$ with the first three terms in the following expansion 
\[
 \td u = B(u,u) = B(\td u,\td u) + B(\td u,P_0)+B(P_0,\td u).
\]
The proof of the lemma also illustrates the barriers to getting the above estimate in $L^{3/2,1}$. Importantly, this failure is one reason we do not achieve the separation rate $t^{3/2}$ in our main theorem. Once we move away from the exponent $3/2$, it is sufficient to consider Lebesgue norms instead of Lorentz norms. Finally, as we will use these lemmas for an iterative procedure involving higher Picard iterates, our lemmas apply to any difference $u- P_k$, not just $\td u = u-P_0$.

 \begin{lemma}\label{lemma.decay} Fix $q\in (3/2 ,3)$, $T>0$ and $k\in \N_0$. Assume $u_0\in L^{3,\I}$ and is divergence free. Let $u$ be a weak $L^{3,\I}$-solution with initial data $u_0$. Then, letting $r =\frac {2q} {2q-3}$,

\[
 \| B(u-P_k,u-P_k) \|_{L^r(0,T;L^{q,1})} \lesssim_{k,q,u_0} T^{\frac 1 2}.
\]
\end{lemma}

The above estimate is dimensionless.
By Lorentz space embeddings we trivially infer
\[
 \| B(u-P_k,u-P_k) \|_{L^r(0,T;L^{q,\beta})} \lesssim_{k,q,u_0} T^{\frac 1 2},
\]
for every $\beta \in (1,\I]$. This includes $L^r(0,T;L^q)$ when $\beta=q$. In our application, we will only use the $L^q$ version of this. However, the proof for the full scale of Lorentz spaces is no harder and we therefore include it in case it is useful elsewhere.

\begin{proof} 
By Yamazaki \cite{yamazaki},
\EQ{\label{ineq: Yamazaki}
 \| B(u-P_k,u-P_k) \|_{L^{q,1}}&\lesssim \int_0^t \frac 1 {(t-s)^\frac 1 2} \| (u-P_k)^2\|_{L^{q,1}}(s)\,ds
 \\&\lesssim \int_0^t \frac 1 {(t-s)^\frac 1 2} \| u-P_k\|_{L^{2q,2}}^2(s)\,ds.
}
 Recall the extension of the Gagliardo-Nirenberg inequality to the Lorentz scale \cite[Corollary 2.2]{DDN} which states
\EQ{\label{GN}
 \| f\|_{L^{\td p,\be}} \lesssim_{\td p, \td q, \be} \|f\|_{L^{\td q,\I}}^{\th}\|\nb f\|_{L^{2}}^{1-\th},
}
for $\be>0$ and 
\[
 \frac 1 {\td p} = \frac{\th} {\td q}+(1-\th) \bigg(\frac 1 2 -\frac 1 3\bigg),
\]
where $1\leq \td q<\td p< \I$ and
$3/2 -3/{\td p} <1$.
Let $\td p = 2q$ and $\td q=2$. These satisfy the above conditions because $q<3$. Then, $\th$ is given by 
\[
 \frac 3 {2q} -\frac 1 2= \th,
\]
and, provided $1<q<3$, the other conditions above are met.
To summarize,
\[
 \|f\|_{L^{2q,2}} \lesssim_{q} \|f\|_{L^{2,\I}}^{\th}\|\nb f\|_{L^{2}}^{1-\th}\lesssim_{q} \|f\|_{L^2}^{\th}\|\nb f\|_{L^2}^{1-\th},
\]
where we used the continuous embedding $L^2 \subset L^{2,\I}$.
Returning to our main estimate, this gives
\EQ{\| B(u-P_k,u-P_k) \|_{L^{q,1}}(t) 
 &\lesssim_{q} \int_0^t \frac 1 {(t-s)^\frac 1 2}\|u-P_k\|_{L^2}^{2\th}\|\nb(u-P_k)\|_{L^2}^{2(1-\th)}\,ds 
 \\&\lesssim_{k,q,u_0} t^{\frac {\th} 2} \int_0^t \frac 1 {(t-s)^\frac 1 2} \|\nb( u-P_k)\|_{L^2}^{2(1-\th)}\,ds,
}
where we used the fact that \eqref{ineq.BSSdecay} applies also to $u-P_k$ as a consequence of \cite[Lemma 2.2]{AB},\footnote{We will use the fact several times and presently elaborate on how it follows from \cite[Lemma 2.2]{AB}. The bounds \cite[(2.36)-(2.39)]{AB} allow us to extend \eqref{ineq.BSSdecay} to $u-P_k$ for $k>0$. Note that $\| \nb (P_{k+1}-P_k)\|_{L^2(0,T;L^2)}\lesssim T^{ 1/4}$ is not mentioned in \cite[(2.39)]{AB} but, upon inspecting the proof, it also holds as a consequence of the energy estimate for the Stokes equation and the above listed bounds.
} in which case the suppressed constant accrues a dependence on $k$. 
Note that for $t\in (0,T)$,
\[
 \int_0^t \frac 1 {(t-s)^\frac 1 2} \|\nb( u-P_k)\|_{L^2}^{2(1-\th)} \,ds 
 \leq \int_{\R} \frac 1 {|t-s|^{\frac 1 2}} \|\nb (u-P_k)\|_{L^2}^{2(1-\th)}(s) \chi_{(0,T)}(s)\,ds,
\]
and the right hand side can be viewed as $I_{\frac 12}( \|\nb (u-P_k)\|_{L^2}^{2(1-\th)} \chi_{(0,T)})$ where $I_{\frac 12}$ is a Riesz potential in 1D.
The Hardy-Littlewood-Sobolev inequality states that 
\[
 \left\| I_{\frac 1 2} \|\nb (u-P_k)\|_{L^2}^{2(1-\th)} \chi_{(0,T)}\right\|_{L^r(\R)} 
 \lesssim_{r} \left\| \|\nb (u-P_k)\|_{L^2}^{2(1-\th)} \chi_{(0,T)}\right\|_{L^{\td r}(\R)},
\]
where 
\[
 \frac 1 r = \frac 1 {\td r}-\frac 1 2.
\]
The selection
\[
 \td r = \frac 1 {1-\th}; \qquad r=\frac 2 {1-2\th},
\]
is valid for the Hardy-Littlewood-Sobolev inequality provided $3/2 <q$.\footnote{If $q=3/2$, then $\th = 1/2$ and $r=\I$, which is not permitted in the Hardy-Littlewood-Sobolev inequality. 
} Letting $r = \frac {2q}{2q-3}$ and putting the above observations together leads to 
\EQ{\label{ineq:lemresult}
 \| B(u-P_k,u-P_k)\|_{L^r (0,T;L^{q,1})} 
 &\lesssim_{k,q,u_0} T^\frac {\th}2\big\| \|\nb (u-P_k)\|_{L^2}^{2(1-\th)} \chi_{(0,T)} \big\|_{L^{\td r}(\R)}
 \\&\lesssim_{k,q,u_0} T^\frac {\th}2 \| \nb (u-P_k)\|_{L^2(0,T;L^2)}^{\frac 2 {\td r}}
 \\&\lesssim_{k,q,u_0} T^\frac {\th}2 T^{\frac 1 {2\td r}}
 =T^{\frac 1 2},
}
where we used the extension of \eqref{ineq.BSSdecay} to $u-P_k$ again.
\end{proof}

We prove a similar result for $B(P_k,u-P_k)$. This requires the well known fact that if $u_0\in L^{3,\I}$, then $P_k$ is in the scaling invariant Kato classes for $q\in (3,\I]$, i.e.,
\EQ{\label{KatoClassInclusion}
\| P_k\|_{\mathcal K_q} :=\sup_{0<t<\I} t^{\frac 1 2-\frac 3{2q}} \|P_k\|_q(t) \lesssim_{k,u_0} 1.
}
To check this, note that the above property is immediate for $P_0$ by the embedding $L^{3,\I}\subset \Bp$, for $3<p\leq \I$, and the fact that $\| u_0\|_{\Bp} \sim \|P_0 \|_{\mathcal K_p}$ \cite{BCD}. Then, by the standard bilinear estimate (see the original papers \cite{FJR,Kato} or \cite[Ch.~5]{Tsai-book}),
\EQ{\label{mild.est}
\| B(f,g)\|_{L^p}(t) \lesssim \int_0^t\frac 1{(t-s)^{\frac 1 2}s^{ \frac 3 2 (\frac 1q-\frac 1 p)}} \| f\otimes g\|_{L^q}(s) \,ds,
}
we have 
\EQ{
\| B(P_{k},P_k)\|_{L^\I}(t) &\lesssim 
\int_0^t \frac 1 {(t-s)^{\frac 1 2+\frac 3 {2q}}} \| P_{k-1}\|_{L^q}(s) \|P_{k-1}\|_{L^\I}(s)\,ds
\\&\lesssim \int_0^t \frac 1 {(t-s)^{\frac 1 2+\frac 3 {2q}} s^{1-\frac 3 {2q}}} \|P_{k-1}\|_{\mathcal K_q} \|P_{k-1}\|_{\mathcal K_\I} \,ds
\\&\lesssim t^{-\frac 1 2} \|P_{k-1}\|_{\mathcal K_q} \|P_{k-1}\|_{\mathcal K_\I},
}
and 
\EQ{
\| B(P_{k},P_k)\|_{L^q}(t) &\lesssim 
\int_0^t \frac 1 {(t-s)^{\frac 1 2 }} \| P_{k-1}\|_{L^q}(s) \|P_{k-1}\|_{L^\I}(s)\,ds
\\&\lesssim \int_0^t \frac 1 {(t-s)^{\frac 1 2 } s^{1-\frac 3 {2q}}} \|P_{k-1}\|_{\mathcal K_q} \|P_{k-1}\|_{\mathcal K_\I} \,ds
\\&\lesssim t^{-\frac 1 2+\frac 3 {2q}} \|P_{k-1}\|_{\mathcal K_q} \|P_{k-1}\|_{\mathcal K_\I}.
}
The claim \eqref{KatoClassInclusion} follows from the above observations by induction.

\begin{lemma}\label{lemma.decay2} Fix $q\in (3/2 ,3)$, $T>0$ and $k\in \N_0$. Assume $u_0\in L^{3,\I}$ and is divergence free. Let $u$ be a weak $L^{3,\I}$-solution with initial data $u_0$. Then, letting $r =\frac {2q} {2q-3}$, 
\[
 \| B(P_k,u-P_k) \|_{L^r(0,T;L^{q,1})}+ \| B(u-P_k,P_k) \|_{L^r(0,T;L^{q,1})} \lesssim_{k,q,u_0} T^{\frac 1 2}.
\]

\end{lemma}
\begin{proof} 
Although the bilinear operator $B$ is not symmetric, our estimates are identical for the two terms in the lemma. Hence we only consider $B(P_k,u-P_k)$.
By Yamazaki \cite{yamazaki} and O'Neil's inequality,
\EQ{\| B(P_k,u-P_k) \|_{L^{q,1}}&\lesssim \int_0^t \frac 1 {(t-s)^\frac 1 2} \| P_k(u-P_k)\|_{L^{q,1}}(s)\,ds
 \\&\lesssim \int_0^t \frac 1 {(t-s)^{\frac 1 2}} \bigg(\|P_k\|_{L^{2q,\I}}^2 + \| u-P_k\|_{L^{2q,1}}^2\bigg)\,ds
 \\&\lesssim \int_0^t \frac 1 {(t-s)^{\frac 1 2}} \bigg(\|P_k\|_{L^{2q}}^2 + \| u-P_k\|_{L^{2q,1}}^2\bigg)\,ds.
}
Noting that we may choose $\beta=1$ in the extension of the Gagliardo-Nirenberg inequality \eqref{GN} to the Lorentz scale, we have the desired result for the $u-P_k$ term by the work done between \eqref{ineq: Yamazaki} and \eqref{ineq:lemresult} in the proof of \cref{lemma.decay}. 

We then consider $P_k$ in $L^{2q}.$
By the membership of $P_k$ in the Kato class,
\EQ{\label{ineq:PkKq}\int_0^t \frac 1 {(t-s)^{\frac 12}} \|P_k\|_{L^{2q}}^2 \, ds
 &\lesssim \int_0^t \frac 1 {(t-s)^{\frac 1 2}s^{1-\frac 3 {2q}}} \|P_k\|_{\mathcal K_{2q}}^2\, ds \\
 &\lesssim_{k,q,u_0} t^{\frac 3 {2q}-\frac 1 2}.
}
Then, 
\[
 \| B(P_k,u-P_k) \|_{L^r(0,T;L^{q,1})} \lesssim_{k,q,u_0} \left(\int_0^T \left(t^{\frac 3 {2q}-\frac 1 2}\right)^r \,dt\right)^{\frac 1r} \lesssim_{k,q,u_0} T^{\frac 3 {2q}-\frac 1 2+\frac 1 r}\lesssim_{k,q,u_0} T^{\frac 1 2},
\]
where we used the definitions of $r$ and $\th$.
\end{proof}
Together, the above two lemmas lead to the following corollary.
 
\begin{corollary}\label{cor:decay}
Fix $q\in (3/2 ,3)$, $T>0$ and $k\in \N_0$. Assume $u_0\in L^{3,\I}$ and is divergence free. Let $u$ be a weak $L^{3,\I}$-solution with initial data $u_0$. Then, letting $r =\frac {2q} {2q-3}$, we have
\[
 \| u-P_k \|_{L^r(0,T;L^{q})} \lesssim_{k,q,u_0} T^{\frac 1 2}.
\]
\end{corollary}
\begin{proof}
This is immediate given Lemmas \ref{lemma.decay} and \ref{lemma.decay2} and the fact that \[u-P_k = B(u-P_{k-1},u-P_{k-1})+B(P_k,u-P_{k-1})+B(u-P_{k-1},P_k),\]
for $k\geq 1$.
\end{proof}

Our next lemma is a technical statement about the decay at $t=0$ of the heat semigroup. 
\begin{lemma}\label{lem:heatL321} Let $B=B_{R}(x_0)$ and $B' := B_{r}(x_0)$ where $0<r<R<\I$. Then, for $0<t<\I$,
\EQ{\left\|\|e^{-\frac{|x-y|^2}{4t}}(1-\chi_{B})\|_{L^{\frac 3 2,1}_y}\right\|_{L^\I_x(B')}\lesssim_{R,r} e^{\frac{-(R-r)^2}{4t}}.}
\end{lemma}

\begin{proof}
{First, assume without loss of generality that $x_0=0$.} Then, letting $x\in B'$,
\EQ{\|e^{-\frac{|x-y|^2}{4t}}(1-\chi_{B})\|_{L^{\frac 3 2,1}_y}
 &= \frac 3 2 \int_0^\I \mu \left\{ y:e^{-\frac{|x-y|^2}{4t}}(1-\chi_{B}(y)) \ge s\right\}^{\frac 2 3}\,ds,
 }
 where $\mu$ is Lebesgue measure. 
{Note that the above set can be written as \[A(x)=\{y:|x-y|\le \sqrt{-4t\ln(s)},\, |y|>R\}=B(x,(-4t\ln(s))^{\frac 1 2}) \setminus B_R(0),\] which is well-defined because $t\ge 0$ and $s\le 1$. Then, \EQ{\left\|\|e^{-\frac{|x-y|^2}{4t}}(1-\chi_{B}(y))\|_{L^{\frac 3 2,1}_y}\right\|_{L^\I_x(B')}
 &\lesssim \left\|\int_0^\I 
 \mu (A(x))^{\frac 2 3}\, ds\right\|_{L^\I_x(B)}\\
 &\lesssim \int_0^{e^{\frac{-(R-r)^2}{4t}}} |-4t\ln(s)| \, ds\\
 &\lesssim 4t \left(e^{\frac{-(R-r)^2}{4t}}\frac{(R-r)^2}{4t}+e^{\frac{-(R-r)^2}{4t}}\right) \\
 &\lesssim_{R,r} e^{\frac{-(R-r)^2}{4t}}.}
 } 
\end{proof}

The above lemma leads to a local \textit{a priori} inclusion for Picard iterates.

\begin{lemma}\label{lemma:Pkq} Let $B=B_R(x_0)$ and $B' = B_r(x_0)$ where $0<r<R<\I$. Let $u_0\in L^{3,\I}$ with $u_0|_{B}\in L^{q}(B)$, for some $3< q\leq \I$. It follows that $P_k \in L^\I(0,\I; L^q(B'))$. 
\end{lemma}

\begin{proof} Note that for any $\tau>0$, $\sup_{\tau<t<\I} \|P_k\|_q \lesssim_{\tau,k} \|u_0\|_{L^{3,\I}}$ due to the fact that $P_k\in \mathcal K_q$ when $q>3$. We therefore only need to prove the inclusion for a short period of time. Let $\{B_k\}$ be a collection of concentric balls about $x_0$ of radii $\al^{k+1}R$, for some $\al\in (0,1)$. Fix $k_0\in \mathbb N_0$. Choose $\alpha$ so that $r = \al^{k_0+1}R$.

For $P_0=e^{t\Delta}u_0$ we have
\EQ{\|P_0\|_{L^q(B_0)}(t) \label{A1}
 &= \left\|\int_{\R^3}\frac{1}{t^{\frac 3 2}}e^{-\frac{|x-y|^2}{4t}} u_0(y) \, dy\right\|_{L^q(B_0)}\\
 &= \left\|\left(\int_{B^c}+\int_{B}\right)t^{-\frac 3 2} e^{-\frac{|x-y|^2}{4t}} u_0(y) \, dy\right\|_{L^q(B_0)} \\
&\lesssim \|u_0\|_{L^{3,\I}}t^{-\frac 3 2}\left\|\|e^{-\frac{|x-y|^2}{4t}}(1-\chi_{B}(y))\|_{L^{\frac 3 2,1}_y}\right\|_{L^q_x(B_0)}\\&\quad +\left\|e^{t\Delta}(\chi_B(y) u_0)\right\|_{L^q(\R^3)}.
}
For the far-field term, by \cref{lem:heatL321},
\EQ{ \label{A2}\left\| \|e^{-\frac{|x-y|^2}{4t}}(1-\chi_{B}(y))\|_{L^{\frac 3 2,1}_y}\right\|_{L^q_x(B_0)}&\lesssim_{R,\al,q}
\left\|\|e^{-\frac{|x-y|^2}{4t}}(1-\chi_{B}(y))\|_{L^{\frac 3 2,1}_y}\right\|_{L^\I_x(B_0)}\\&\lesssim_{R,\al} e^{\frac{-(R(1-\al))^2}{4t}}.}
For the near field term,  
\EQ{
\|e^{t\Delta}(u_0\chi_B )\|_{L^q(\R^3)} \lesssim \| u_0 \chi_B\|_{L^q(\R^3)}\lesssim \|u_0\|_{L^q(B)}.
} 
Therefore, 
\EQ{
 \big\| \|P_0\|_{L^q(B')}\big\|_{L^\I_t}
 \lesssim_{R,\al} \|u_0\|_{L^{3,\I}}+ \|u_0\|_{L^q(B)}.
}
If $k_0=0$ then we are done. 

If $k_0>0$ then we use induction. 
Observe that
\EQN{
 B(P_{k-1},P_{k-1}) =B(P_{k-1}\chi_{B_{k-1}},P_{k-1})+B(P_{k-1}(1-\chi_{B_{k-1}}),P_{k-1}).
}
For the first part,
\EQN{
 \|B(P_{k-1}\chi_{B_{k-1}},P_{k-1})\|_{L^q(B_k)}(t) &\lesssim_q\|P_k\|_{\mathcal K_\I} \int_0^t \frac 1 {(t-s)^{\frac 1 2} s^{\frac 1 2}} \| P_{k-1}\|_{L^q(B_{k-1})}(s) \,ds
 \\&\lesssim_{k,q} \|u_0\|_{L^{3,\I}}\| P_{k-1}\|_{L^q(B_{k-1})}(t) .
}
For the other part, by the pointwise estimate for the kernel $K$ of the Oseen tensor (see \cite{VAS,Tsai-book}),
\EQ{\label{oseen.estimate}
|D_k^{(m)} K(x,t)| \lesssim_m \frac 1 {(|x|+\sqrt{t})^{3+|m|}},
}
we have 
\EQN{
 \|B(P_{k-1}(1-\chi_{B_{k-1}}),P_{k-1})\|_{L^q(B_k)}(t)&\lesssim\left\|\int_0^t \int_{B_{k-1}^c} \frac {P_{k-1}\otimes P_{k-1}(y,s)} {(|x-y|+\sqrt{t-s})^4} \,dy\,ds\right\|_{L^q(B_k)}
 \\&\lesssim_{R,\al,k,q} \| |\cdot|^{-4}\|_{L^{2}(|\cdot|>R(\al^k-\al^{k+1}) )} \int_0^t \|P_{k-1}\|_{L^{4}}^2\,ds 
 \\& \lesssim_{R,\al,k,q} \int_0^t s^{(-1+\frac 3 {4})}\,ds \lesssim_{R,\al,k,q} t^{\frac 34},
} 
where we used the membership of $P_{k-1}$ in the Kato class $\mathcal K_4$.

We know by our base case that $P_0$ is in $L^\I(0,\I; L^q(B_0))$. We have just shown $B(P_{k-1},P_{k-1})\in L^\I(0,\I; L^q(B_0))$ whenever $P_{k-1}$ is in $ L^\I L^q(B_{k-1})$. Hence,
\[
 P_k = P_0-B(P_{k-1},P_{k-1})\in L^\I(0,\I; L^q(B_k)).
\]
This extends up to $k_0$ and so $P_{k_0}\in L^\I(0,\I; L^q(B'))$.
\end{proof}

\begin{remark}\label{Remark.heat}
Under the assumptions of Lemma \ref{lemma:Pkq} and
by classical estimates for the heat semi-group,
\[
\|e^{t\Delta} ( u_0\chi_B) \|_{L^\I(B')} (t)\lesssim t^{-\frac 3 {2q}} \|u_0\|_{L^q(B)}.
\]
Note that, combining \eqref{A1} and \eqref{A2},  
\EQ{
\|e^{t\Delta} (u_0(1-\chi_B)) \|_{L^\I(B')}(t) \lesssim_{T} \|u_0\|_{L^{3,\I}}t^{-\frac 3 {2q}},
}provided $t<T$, for any given time $T$.
Hence, 
\[
\| e^{t\Delta}u_0 \|_{L^\I(B')}(t)\lesssim_{u_0,T} t^{-\frac 3 {2q}}.
\]
\end{remark}

\section{Proof of Theorem \ref{thrm.main2}} 
 
Our foundation for the proof of Theorem \ref{thrm.main2} is the local smoothing result of Jia and \v Sver\'ak \cite{JS}, which we presently restate. Note that $L^2_\uloc$ is the space of uniformly locally square integrable functions and is defined by the norm
\[
\|f\|_{L^2_\uloc}^2 :=\sup_{x_0\in \R^3}\int_{B_1(x_0)}|f|^2\,dx. 
\]
Let $E^2$ denote the closure of $C_c^\I$ in the $L^2_\uloc$ norm.
 $L^{3,\I}$ embeds in $E^2$ (see the appendix of \cite{BT8}). 
Local smoothing as presented below refers to local energy solutions (a.k.a.~local Leray solutions using the terminology of \cite{JS}; see also \cite{BT8,KS,LR}). 
It is straightforward to show that weak $L^{3,\I}$-solutions are local energy solutions.

\begin{theorem}[Local smoothing {\cite[Theorem 3.1]{JS}}]\label{theorem:JSlocalsmoothing} 
 Let $u_0\in E^2$ be divergence free. Suppose $u_0|_{B_2(0)}\in L^p(B_2(0))$ with $\|u_0\|_{L^p(B_2(0))}<\I$ and $p>3$. Decompose $u_0 =U_0+U_0'$ with $\div U_0 =0,\,U_0|_{B_{4/3}} = u_0,\, \operatorname{supp} U_0 \Subset B_2(0)$ and $\|U_0\|_{L^p(\R^3)}<C(p,\|u_0\|_{L^p(B_2(0)})$. Let $U$ be the locally-in-time defined mild solution to \eqref{eq.NS} with initial data $U_0$. Then, there exists a positive $T=T(p,\|u_0\|_{L^2_\uloc},\|u_0\|_{L^p(B_2(0))})$ such that any local energy solution $u$ with data $u_0$ satisfies 
 \EQ{
 \|u-U\|_{C^\gamma_{par}(\overline {B}_{\frac{1}{2}}\times [0,T])}\le C(p,\|u_0\|_{L^p(B_2(0))},\|u_0\|_{L^2_\uloc}),
 } 
 for some $\gamma=\gamma(p)\in(0,1)$.
\end{theorem}

See also \cite{BP2020,KMT,KMT2,Kwon} for more recent work on local smoothing which allows locally critical data which is also locally small; the above statement on the other hand is for locally sub-critical data. The dependence on $\|u_0\|_{L^2_\uloc}$ can be replaced with $\|u_0\|_{L^{3,\I}}$, which is why $L^2_\uloc$ is not mentioned in Theorem \ref{thrm.main2}.

\begin{proof}[Proof of \cref{thrm.main2}]
Without loss of generality, assume $B:=B_2(x_0)$ is centered at $x_0=0$. Assume $u_0|_B \in L^p(B)$. Let $U_0$ be a localization of the data to $B$ such that $u_0=U_0$ in $B_{4/3}(0) \subset B$, $\supp U_0 \Subset B$. This is done via a Bogovskii map \cite{GGP} as per the decomposition in \cref{theorem:JSlocalsmoothing}. Let $U$ be the locally-in-time defined mild solution to \eqref{eq.NS} with data $U_0$. Define $\{B_k\}_{k=0}^\I$ to be a collection of nested balls centered at $0$, with radii $\al^{k}/2$, for some $\al\in(0,1)$ to be specified later.   Then, recalling $P_0=e^{t\Delta}u_0$,
\EQ{ 
|u-P_0|(x,t)  \le& |u-U|(x,t) +|U-e^{t\De}U_0|(x,t) +|e^{t\De}(U_0-u_0)|(x,t) \\
	=:& I_1(x,t)+I_2(x,t)+I_3(x,t).
}%
In the definition of $C^\gamma_{par}(\overline {B}_{\frac{1}{2}}\times [0,T])$, the exponent in the time-variable modulus of continuity is $\ga /2$.
By local smoothing \eqref{theorem:JSlocalsmoothing} and the fact that $\|u_0\|_{L^2_\uloc}\lesssim \|u_0\|_{L^{3,\I}}$, there exists $T = T(p,u_0)>0 $ so that \[I_1(x,t) \lesssim_{p,u_0} t^{\frac \ga 2},\] for some $\ga=\ga(p) \in (0,1)$, $x\in B_0$ and  $0<t<T$. 

For $I_2$, by \eqref{mild.est},  for any $p\in (3,\I]$ and $0<t<T$
\EQ{
 I_2(x,t) &\leq \| B(U,U)\|_{L^\I(\R^3)}(t)\\&\lesssim t^{\frac 1 2 -\frac 3 {2p}} \|U\|_{L^\I (0,T;L^p)}^2 \lesssim t^{\frac 1 2 -\frac 3 {2p}} \|U_0\|_{L^p}^2,
} 
where we possibly re-define $T$ to make it smaller than the timescale of existence for the  strong solution to \eqref{eq.NS}, i.e.~$T\lesssim \| U_0 \|_{L^p}^{-2p/(p-3)}$, and the time-scale coming from Theorem \ref{theorem:JSlocalsmoothing}.

Noting that $U_0-u_0=0$ in $B_{4/3}$, the last part, $I_3$, is broken into integrals over a shell and a far-field region as follows, 
\EQ{I_3(x,t) 
 \lesssim& \left(\int_{\frac 4 3\le|y|< 2}+\int_{|y|\ge 2}\right)t^{-\frac 3 2} e^{-\frac{|x-y|^2}{4t}} |U_0-u_0|(y) \, dy
 =:I_{31}(x,t)+I_{32}(x,t).
 }
For $I_{31}$, using the fact that $U_0$ was solved for via a Bogovskii map, and therefore $\|U_0\|_{L^p(\R^3)} \lesssim \|u_0\|_{L^p(B)}$, we have for all $0<t<T$ and $x\in B_0$ that
\EQ{
I_{31}(x,t)&\lesssim 
t^{-\frac 3 2} e^{-\frac{(\frac 4 3-\frac 1 2)^2}{4t}} \| U_0 - u_0\|_{L^p (\frac 4 3\le|y|<2)}(t)
\lesssim_{u_0,p} t^{\frac \ga 2}.}
For $I_{32}$, by \cref{lem:heatL321}, the fact that $U_0(y) \equiv 0$ for $|y|\ge 2$, and taking $x\in B_0$ and $0<t<T$, we have that
\EQ{I_{32}(x,t)
 &\lesssim \int_{|y|\ge 2}t^{-\frac 3 2} e^{-\frac{|x-y|^2}{4t}} |u_0|(y) \, dy \\&\lesssim t^{-\frac 3 2} \|u_0\|_{L^{3,\I}}\left\|\|e^{-\frac{|x-y|^2}{4t}}(1-\chi_{B}(y))\|_{L^{\frac 3 2,1}_y}\right\|_{L^\I_x(B_0)}\\
&\lesssim_{u_0} t^{-\frac 3 2}e^{\frac{-(2-\frac 1 2)^2}{4t}}
\lesssim_{p,u_0} t^{\frac \ga 2}.}
 Therefore,  
 \[\|u-P_0\|_{L^\I(B_0)}(t) \lesssim_{p,u_0} t^{\min\{\frac \ga 2, \frac 12 - \frac 3 {2p}\}},\]
where the dependence on $u_0$ is via the quantities $\|u_0\|_{L^p(B)}$ and $\|u_0\|_{L^{3,\I}}$.

We inductively extend this estimate to higher Picard iterates. Fix $\sigma$ as in the statement of the theorem.
Recursively define the sequence $\{a_k\}$ by $a_{k+1} =\min\left\{\sigma, 1/2 -3/(2p)+a_k\right\}$ and $a_0 = \min\{\ga/2, 1/2 -3/(2p)\}$.
Assume for induction that
\[
 \|u-P_k\|_{L^\I(B_k)} \lesssim _{k,\al, p,u_0} t^{a_k},
\]
for $0<t<T$ and the dependence on $u_0$ is via the same quantities as above.
Note that
\EQ{
 |u-P_{k+1} |(x,t) \leq& |B(u-P_k, u-P_k)| + |B(u-P_k,P_k)|+ |B(P_k,u-P_k)| \\=:& J(x,t)+K(x,t)+L(x,t).
}

We split $J$ further as
\EQ{J(x,t) &\le | B((u-P_k) \chi_{B_k} , u-P_k)|+| B((u-P_k) (1-\chi_{B_k}) , u-P_k)|\\&=: J_{1}(x,t)+J_{2}(x,t).}
For the near-field, $J_1$, we use the inductive hypothesis to obtain for that, for $0<t<T$,
\EQ{\label{ineq:J1}
	\|J_{1}\|_{L^\I(B_{k+1})}(t)&
	\lesssim \int_0^t \frac 1 {\sqrt {t-s}} \| u-P_k \|_{L^\I(B_k)}^2 \,ds\\
	&\lesssim_{k,\al, p,u_0} t^{\frac 1 2 +2a_k} \lesssim_{k,\al, p,u_0} t^{\frac 1 2 -\frac 3{2p} +a_k}.
}
Considering $J_{2}$, for  $0<t<T$, we have by \eqref{oseen.estimate} that
\EQ{\label{ineq:J2}
\|J_2\|_{L^\I(B_{k+1})}(t)&\lesssim \int_0^t 
	\int_{|x-y|>\frac1 2 \al^k-\frac 12\al^{k+1}} \frac 1 {|x-y|^4} |u -P_k|^2(y,s)\,dy\,ds \\&\lesssim \frac t {(\al^{k}-\al^{k+1})^4} \| u -P_k\|_{L^2}^2(t)\lesssim_{\al,k,u_0} t^{\frac 3 2},
}
where we used the version of \eqref{ineq.BSSdecay} for higher Picard iterates \cite{AB}.

The terms $K$ and $L$ are treated identically and we only consider $K$. We begin by further splitting $K$ as
\[K(x,t) \leq | B((u-P_k) \chi_{B_k} , P_k)|+| B((u-P_k) (1-\chi_{B_k}) , P_k)|=: K_1(x,t)+K_2(x,t)
.\] For the near-field $K_1$ and for $0<t<T$ we have 
\[
\| K_1\|_{L^\I(B_{k+1})}(t) \lesssim \int_0^t \frac 1 {(t-s)^{\frac 1 2 + \frac 3 {2p}}} \| u-P_k \|_{L^\I(B_k)}(s) \| P_k\|_{L^p(B_k)}(s)\,ds.
\]
By \cref{lemma:Pkq},  $\sup_{0<t<\I} \| P_k\|_{L^p(B_k)} <\I$. Note that $1/2 + 3/(2p)<1$ precisely if $3<p$. Hence,
\EQ{\label{ineq:K1}
\| K_1\|_{L^\I(B_{k+1})}(t)
 &\lesssim_{k,\al, p,u_0} t^{\frac 1 2 - \frac 3 {2p}+a_k},
}
for $0<t<T$ by the inductive hypothesis. 
For the far-field $K_2$, using  \cref{cor:decay} and taking $x\in B_{k+1}$, $0<t<T$ and $q\in (3/2,3)$, we have by \eqref{oseen.estimate} that
\EQ{
K_2(x,t) &\lesssim  \int_0^t \int_{B_k^c} \frac {1} {(|x-y| +\sqrt{t-s})^4} |u-P_k| |P_k| \,dy\,ds 
\\ & \lesssim 
\bigg\| \frac {1-\chi_{B_k}(\cdot)} {|x-\cdot |^4} \bigg\|_{ L^{r'}(0,T; L^{q'})} \|P_k \|_{L^\I(0,T;L^{3,\I})} \| u-P_k\|_{L^r(0,T;L^q)}
\\&\lesssim_{k,q,u_0} t^{\frac 1 {r'}+\frac 1 2},
}
where 
\[
1 = \frac 1 q +\frac 1 {q'} +\frac 1 3\text{ and }
1 = \frac 1 r +\frac 1 {r'}.
\]
Letting
\[
r= \frac {2q} {2q-3},
\]
we have
\[
\frac 1 {r'} = \frac 3 {2q}.
\]
Observe that $1/r'<1$ and
\[
 \lim_{q\to {\frac 3 2}^+} \frac 1 {r'} = 1.
\]
Therefore, for any $\sigma <3/2$, by taking $q>3/2$ sufficiently close to $3/2$,  
\EQ{\label{ineq:K2}
 K_2(x,t)\lesssim_{k,\sigma,u_0,q} t^{\sigma}.
}
Altogether, \eqref{ineq:J1},\eqref{ineq:J2},\eqref{ineq:K1}, and \eqref{ineq:K2} imply that, for $0<t<T$,
\[
 \|u-P_{k+1}\|_{L^\I(B_{k+1})}(t) \lesssim_{k,\al, p,u_0,q} t^{a_{k+1}},
\]
for $k\ge 0$ and any $\sigma<3/2$. 
Note that
\[
a_{k+1} = \min \left\{ \sigma , (k+1)\bigg(\frac 1 2 -\frac 3 {2p}\bigg) +a_0 \right\}.
\]
Choose $k_0$ to be the smallest natural number so that 
\[
k_0\bigg(\frac 1 2 -\frac 3 {2p}\bigg) +a_0 \geq \sigma.
\]
Then, $a_{k_0}=\sigma$ and $a_{k}<a_{k-1}$ for $k=1,\ldots,k_0$.
Choose $\al$ so that $\rho = \al^{k_0}/2$.
It follows that 
\[
 \|u-P_{k_0}\|_{L^\I(B_{\rho}(x_0))}(t) \lesssim_{\rho, p,\sigma,u_0} t^{\sigma}.
\]

Regarding the asymptotic expansion, we observe that for $1\leq k\leq k_0$ and $(x,t)\in B_\rho(x_0)\times (0,T)$, 
\[
u = P_{k_0}+\mathcal O(t^\sigma),
\]
and
\[
|P_{k}-P_{k-1}|(x,t)\le |u-P_k|(x,t)+|u-P_{k-1}|(x,t)=\mathcal O(t^{a_{k-1}}).
\]
Hence,
\EQ{
u(x,t) &= \underbrace{P_{0} + \sum_{k=1}^{k_0} (P_{k}-P_{k-1})(x,t)}_{=P_{k_0}} + \mathcal O(t^\sigma)
\\&=\mathcal O (t^{-\frac 3 {2p}})
+\sum_{k=0}^{k_0-1} \mathcal O(t^{a_k}) + \mathcal O(t^\sigma) = \sum_{k=-1}^{k_0} \mathcal O(t^{a_k}),
}
where we are letting $a_{-1}=-3/(2p)$   and are using Remark \ref{Remark.heat} to obtain the asymptotics for $P_0$.
\end{proof}

\section*{Acknowledgements}

The research of Z.~Bradshaw is supported in part by the Simons Foundation.

\bigskip 
\noindent Zachary Bradshaw, Department of Mathematical Sciences, 309 SCEN,
University of Arkansas,
Fayetteville, AR 72701. \url{zb002@uark.edu}

\bigskip 
\noindent Patrick Phelps, Department of Mathematical Sciences, 309 SCEN,
University of Arkansas,
Fayetteville, AR 72701. \url{pp010@uark.edu}

\end{document}